\documentclass[a4paper,12pt]{article}
\usepackage{amssymb}
\usepackage{latexsym}
\usepackage{amsmath,amsthm,amssymb}
\usepackage{amsfonts}

\newtheorem{tw}{Theorem}[section]

\newtheorem{pro}[tw]{Proposition}
\newtheorem{cor}[tw]{Corollary}

\theoremstyle{definition}

\newtheorem{exa}[tw]{Example}
\newtheorem{rem}[tw]{Remark}

\begin{document}

\begin{center}
{\Large On the waiting time till some patterns occur in i.i.d. sequences}
\end{center}

\begin{center}
{\sc Urszula Ostaszewska, Krzysztof Zajkowski}
\footnote{The authors are supported by the Polish National Science Center, Grant no. DEC-2011/01/B/ST1/03838}\\
Institute of Mathematics, University of Bialystok \\ 
Akademicka 2, 15-267 Bialystok, Poland \\ 
uostasze@math.uwb.edu.pl \\ 
kryza@math.uwb.edu.pl 
\end{center}

\begin{abstract}
In this paper we present some general solution of the system of linear equations formed by Guibas and Odlyzko in Th.3.3 \cite{Gui}.
We derive probabilities for given patterns to be first to appear in random text and the expected waiting time till one of them is observed 
and also till  one of them occurs given it is known which pattern appears first.
\end{abstract}



\section{Introduction}
It is a classical problem in probability theory to study occurrence of patterns in a random text formed by independent realizations of letters chosen from finite alphabet.  Feller in \cite{Fel} considered randomness and recurrent patterns connected with Bernoulli trials. 
Solov'ev  in \cite{Sol} found the formula on the expected waiting time for an appearance of one pattern. Penney in \cite{Pen} proposed a coin-flip game with a coin tossed repeatedly until one of patterns appeared, then this pattern wins. A formula for computing  the odds of winning for two competing pattern was discovered by Conway and described by Gardner \cite{Con}. A martingale approach to the study of occurrence of patterns in repeated experiments was presented by Li \cite{Li}.

Results dealing with the occurences of patterns have been applied in several areas of information theory including  source coding, code synchronization, randomness testing, etc. They are also important 
in molecular biology in DNA analysis and for gene recognition. 

Guibas and Odlyzko in Th.3.3 \cite{Gui} formed the system of linear equations which relates the generating function of probabilities that
given pattern occurs before others and the generating function for  tails distribution of probabilities until some pattern appears.
In this paper we present some general solution of this system (Th.\ref{mTh}). It has a general form but it allows us to obtain formulas on the probability
that a given pattern precedes the remaining ones and a new formula on the expected waiting time until some pattern occurs in a random text (Cor.\ref{wn1} and \ref{wn2}).
For complementary of our presentation we recall in a general context some  facts obtained in \cite{KZ} dealt with the probability-generating functions  but in this paper we focus our  
attention on the expected waiting time and we show how this general solution could be used to obtain conditional expected waiting times (Prop.\ref{stw}). 
We also present on some example how it could be used to prove in a simply way some known result deal with presented topic (Rem.\ref{uw1}).

\section{Waiting time on patterns}
Let $(\xi_n)$ be a sequence of i.i.d. random letters from a finite alphabet $\Omega$. 
For a given pattern (word) $A=(a_1,a_2,...,a_l)\in\Omega^l$ by $A_{(k)}$ and $A^{(k)}$ we will denote subpatterns  formed by first and last $k$ letters 
of $A$, respectively; $A_{(k)}=(a_1,a_2,...,a_k)$ and $A^{(k)}=(a_{l-k+1},a_{l-k+2},...,a_l)$. For two patterns $A$ and $B$  we will denote
$[A=B]=1$ if $A=B$ and $[A=B]=0$ if not. Now we define a correlation polynomial $w_A^B$ of $A$ and $B$ (of the lengths $l$ and $m$, respectively) as
$$
w_A^B(s)=\sum_{k=1}^{\min\{l,m\}}[A_{(k)}  =  B^{(k)}]Pr(A^{(l-k)})s^{l-k};
$$
if $l=\min\{l,m\}$ then in the above sum for the index $k=l$ we assume that $P(A^{(0)})=1$.
\begin{exa}
Let $\Omega=\{H,T\}$ and $(\xi_n)$ be a sequence of i.i.d. letters in $\Omega$  with the distribution 
$$
Pr(\xi_n=H)=p\quad{\rm and}\quad Pr(\xi_n=T)=q=1-p.
$$
Consider two patterns $A=THH$ and $B=THTH$. Then the correlation polynomials have the following forms:
$$
w_A^B(s)=ps,\;\; w_B^A(s)=0,\;\; w_A^A(s)=1\;\;{\rm and}\;\;w_B^B(s)=pqs^2+1. 
$$
\end{exa}
Let us note that the correlation polynomials up to some differences in their definition  can be used to investigate the appearances of patterns
not only in Bernoulli trials but also in Markov chains (see for instance \cite{Reg}).

Now we briefly recall results presented by Guibas and Odlyzko in \cite[sec.3]{Gui}.
Consider a set of $m$ patterns (words)  $A_i$ ($1\le i \le m$) of lengths $l_i$, respectively. 
We assume that the set of patterns is reduced
that is none of the patterns contains any other as a subpattern.

Let $\tau_i$ denote the stopping time until $A_i$ occurs and $\tau$ be the stopping time till some of considered patterns is observed, i.e. $$\tau=\min\{\tau_i:\;1\le i \le m\}.$$
 Let $p_{n}$ and $p^{A_i}_n$ be probabilities
$Pr(\tau=n)$ and $Pr(\tau=\tau_i=n)$, respectively. %Notice that $p_n=\sum_{i=1}^m p^{A_i}_n$. 
Let $g_\tau(s)$ denote $E(s^\tau)=\sum_{n=0}^\infty p_ns^n$ the probability generating function of random variable $\tau$   and $g_\tau^{A_i}$ be the generating functions $\sum_{n=0}^\infty p^{A_i}_ns^n$, $1\le i \le m$. 
%$g_\tau^{A_i}(s)=\sum_{n=0}^\infty p^{A_i}_ns^n$. 
Since $p_n=\sum_{i=1}^m p^{A_i}_n$ we have that $g_\tau=\sum_{i=1}^m g_\tau^{A_i}$. 
By $Q_\tau$ we denote the generating function for tails probabilities of $\tau$, i.e. $Q_\tau(s)=\sum_{n=0}^\infty q_ns^n$, where
$q_n=Pr(\tau>n)=\sum_{k>n}p_k$.

Let $B_n$ be the set of sequences such that any pattern $A_i$ does not appear in the string of the first $n$ letters of these sequences.
Notice that $Pr(B_n)=q_n$.
In the system of $m$ patterns if we add to each initial $n$-string in $B_n$ the word $A_i$ then we must check
if neither it nor other ones appear earlier. 
Since $Pr(B_n)=q_n$  we get the following system of equations
\begin{equation}
\label{req1}
q_nPr(A_i)=\sum_{j=1}^m\sum_{k=1}^{\min\{l_i,l_j\}}[A_{i(k)}=A_j^{(k)}]Pr(A_i^{(l_i-k)})p^{A_j}_{n+k},
\end{equation}
for each $1\le i \le m$, which is compatible with (3.2) in \cite{Gui}.
Notice now that 
\begin{equation}
\label{req2}
q_n=q_{n+1}+\sum_{j=0}^m p_{n+1}^{A_j}
\end{equation}
(see (3.1) in \cite{Gui}).
Multiplying (\ref{req1}) by $s^{n+l_i}$ and (\ref{req2}) by $s^n$ and summing from $n=0$ to infinity we obtain the following system of linear equations
\begin{equation}
\label{eq1}
\left\{
\begin{array}{ccl}
(1-s)Q_\tau(s)+\sum_{j=1}^m g_\tau^{A_j}(s) & = & 1\\
Pr(A_i)s^{l_i}Q_\tau(s)-\sum_{j=1}^m w_{A_i}^{A_j}(s)g_\tau^{A_j}(s)  & = & 0\quad (1\le i \le m).
\end{array}
\right.
\end{equation}
\begin{rem}
In our opinion the definition of  the correlation polynomial (second line from above page 195
\cite{Gui})  appearing in the system of linear equation formed in Th.3.3 \cite{Gui}  should be of the following form
$$
c_{GH}(z)=\sum_{r\in GH} z^{r-1}\frac{Pr(h_{r+1},...,h_{|H|})}{Pr(H)}=\sum_{r\in GH}\frac{z^{r-1}}{Pr(h_1,...,h_r)}=\sum_{r\in GH}\frac{z^{r-1}}{Pr(h_1)\cdot...\cdot Pr(h_r)}.
$$
Then substituting $z=\frac{1}{s}$ we get the equivalence of the system of linear equation in Th.3.3 \cite{Gui} and this one described in (\ref{eq1}).
\end{rem}

\begin{tw}
\label{mTh}
The solution of the system of linear equations (\ref{eq1}) has the following form
\begin{equation}
\label{pgf}
g_\tau^{A_i}(s)=\frac{\det \mathcal{B}^i(s)}{\sum_{j=1}^{m}\det \mathcal{B}^j(s)+(1-s)\det \mathcal{B}(s)}\quad(1\le i \le m)
\end{equation}
and 
\begin{equation}
\label{tgf}
Q_\tau(s)=\frac{\det \mathcal{B}(s)}{\sum_{j=1}^{m}\det \mathcal{B}^j(s)+(1-s)\det \mathcal{B}(s)},
\end{equation}
where $\mathcal{B}$ denotes a matrix formed by correlations polynomials $w_{A_i}^{A_j}$, i.e.
\begin{equation*}
\mathcal{B}(s)=
\begin{bmatrix}
w_{A_i}^{A_j}(s)
\end{bmatrix}
_{1\le i,j \le m}
\end{equation*}
and $\mathcal{B}^j(s)$ is the matrix arisen by replacing the $j$-th column of $\mathcal{B}(s)$ by  the column vector 
$[P(A_i)s^{l_i}]_{1\le i \le m}$.
\end{tw}
\begin{proof}
Leading $Q_\tau(s)$ out of the first equation of the system (\ref{eq1}) and substituting it into the remaining ones 
we obtain an equivalent system of linear equations of the form
$$
P(A_i)s^{l_i}=\sum_{j=1}^m g_\tau^{A_j}(s)[P(A_i)s^{l_i}+(1-s)w_{A_i}^{A_j}(s)]\quad (1\le i \le m).
$$
Let $\mathcal{A}$ denote the coefficient matrix of the above system, i.e. 
$$
\mathcal{A}(s) =
\begin{bmatrix}
P(A_i)s^{l_i}+(1-s)w_{A_i}^{A_j}(s)
\end{bmatrix}
_{1\le i,j \le m}.
$$ 
Notice that because $w_{A_i}^{A_i}(0)=1$ and $w_{A_i}^{A_j}(0)=0$ for $i\neq j$ then $\mathcal{A}(0)$ is the identity matrix.
Since  $\det\mathcal{A}(0)=1$ and $\det\mathcal{A}(s)$ is a polynomial, $\det\mathcal{A}(s)\neq 0$ on some neighborhood of zero.
It means that on this neighborhood there exist solution of the system.

Because the determinant 
of matrices $m\times m$ is a $m$-linear functional with respect to columns (equivalently to rows) then one can check that 
$$
\det \mathcal{A}(s)=(1-s)^m\det \mathcal{B}(s) +(1-s)^{m-1}\sum_{j=1}^m \det \mathcal{B}^j(s).
$$
If now similarly $\mathcal{A}^j(s)$ denotes the matrix formed by replacing the $j$-th column of $\mathcal{A}(s)$ by  the column vector $[P(A_i)s^{l_i}]_{1\le i \le m}$ then the determinant's
calculus gives that $\det \mathcal{A}^j(s)=(1-s)^{m-1}\det \mathcal{B}^j(s)$.  By the  Cramer's rule we obtain
$$
g_\tau^{A_i}(s)=\frac{\det\mathcal{A}^i(s)}{\det\mathcal{A}(s)}=\frac{\det \mathcal{B}^i(s)}{\sum_{j=1}^m \det \mathcal{B}^j(s)+(1-s)\det \mathcal{B}(s) }
$$
for $1\le i \le m$.

Substituting the functions $g_\tau^{A_i}$ into the first equation of the system (\ref{eq1}) one can calculate that
$$
Q_\tau(s)=\frac{\det \mathcal{B}(s)}{\sum_{j=1}^{m}\det \mathcal{B}^j(s)+(1-s)\det \mathcal{B}(s)}.
$$

\end{proof}

Notice that the probability-generating function $g_\tau^{A_i}(s)=\sum_{n=0}^\infty p_n^{A_i}s^n$ is well define on the interval $[-1,1]$ for sure
(it is an analytic function on $(-1,1)$). The right hand side of (\ref{pgf}) is a rational function equal to $g_\tau^{A_i}$ on some neighborhood of zero.
By the analytic extension we know that there exists the limit of the right hand side of (\ref{pgf}) by $s\to 1^-$ which is equal to $g_\tau^{A_i}(1)$. Thus we obtain the following 
\begin{cor}
\label{wn1}
The probability $Pr(\tau=\tau_i)$ that the pattern $A_i$ precedes all the remaining $m-1$ patterns is equal to $g_X^{A_i}(1)$, 
that is
\begin{equation}
\label{pro}
Pr(\tau=\tau_i)=\frac{\det \mathcal{B}^i(1)}{\sum_{j=1}^m \det \mathcal{B}^j(1)},
\end{equation}
where the right hand side of the above equality we understand as the limit of (\ref{pgf}) by $s\to 1^-$.
\end{cor}
An application of the above to the generalization of the Conway's formula one can find in \cite{KZ}.

Since $Q_\tau(1)$ is the expected value of $\tau$, we can formulate the another
\begin{cor}
\label{wn2}
The expected waiting time till one of $m$ patterns is observed is given by
\begin{equation}
\label{evQ}
E\tau=\frac{\det \mathcal{B}(1)}{\sum_{j=1}^{m}\det \mathcal{B}^j(1)},
\end{equation}
where the above right hand side is the limit of (\ref{tgf}) by $s\to 1^-$.
\end{cor}
Note that the above formula is also true for one pattern $A\in\Omega^l$ ($m=1$)  since
\begin{equation}\label{solov}
E\tau=\frac{w^A_A(1)}{Pr(A)}=\sum_{k=1}^l\frac{[A_{(k)}=A^{(k)}]}{Pr(A_{(k)})}
\end{equation}
(compare Solov'ev's result in \cite{Sol}). 

Now we can proceed to calculate the expected waiting time for the pattern $A_i$ knowing that it appears as a first, namely the conditional expectation 
of the random variable $\tau$ given the event $\{\tau=\tau_i\}$ that is we prove the following
\begin{pro}
\label{stw}
\begin{equation} 
\label{procond}
E({\tau}|\tau=\tau_i)= E\tau+\frac{1}{Pr(\tau=\tau_i)}\cdot \frac{d}{ds}\Big(\frac{\det\mathcal{B}^i}{\sum_{j=1}^{m}\det \mathcal{B}^j}\Big)(1).
\end{equation}
\end{pro} 
\begin{proof}
Observe that the conditional expectation $E({\tau}|\tau=\tau_i)$ can be expressed in terms of the function $g_\tau^{A_i}$ as follows
\begin{eqnarray}
\label{Contau}
E({\tau}|\tau=\tau_i) & = & \sum_{n=1}^\infty nPr(\tau=n|\tau=\tau_i)=\frac{1}{Pr(\tau=\tau_i)}\sum_{n=1}^\infty nPr(\tau=\tau_i=n) \nonumber\\
\; & = & \frac{1}{Pr(\tau=\tau_i)}\frac{d}{ds}g_\tau^{A_i}(1).
\end{eqnarray}
Differentiating the function $g_\tau^{A_i}$ given by formula (\ref{pgf})  and taking the value at 1 we get
\begin{eqnarray}
\frac{d}{ds}g_\tau^{A_i} (1)&=& \frac{\frac{d}{ds}\det\mathcal{B}^i(1) \sum_{j=1}^{m}\det \mathcal{B}^j (1)- \det\mathcal{B}^i(1)\sum_{j=1}^{m}\frac{d}{ds}\det \mathcal{B}^j (1)}{[\sum_{j=1}^{m}\det \mathcal{B}^j (1)]^2} \nonumber \\
&&
+\frac{\det\mathcal{B}^i(1)}{\sum_{j=1}^{m}\det \mathcal{B}^j (1)} \cdot \frac{\det\mathcal{B}(1)}{\sum_{j=1}^{m}\det \mathcal{B}^j(1)}. \nonumber
\end{eqnarray}
Notice that the first summand can be expressed as $\frac{d}{ds}\Big(\frac{\det\mathcal{B}^i}{\sum_{j=1}^{m}\det \mathcal{B}^j}\Big)(1)$
and the second one, by Corollaries \ref{wn1} and \ref{wn2}, is equal to $Pr(\tau=\tau_i)E\tau$.
Thus, by (\ref{Contau}), we get the formula (\ref{procond}) on $E({\tau}|\tau=\tau_i)$.
\end{proof}

Before we present some examples we would like to show  applications of our results to obtain some known fact contained in \cite{Li}.
\begin{rem}\label{uw1}
Define a number $A_j \ast A_i$ as 
$$
\frac{w_{A_i}^{A_j}(1)}{Pr(A_i)}
=\sum_{k=1}^{\min\{l_i,l_j\}}\frac{[A_{i(k)}=A_j^{(k)}]}{Pr(A_{i(k)})}.
$$
Let us emphasizes that the above coincides with the notation (2.3) in \cite{Li}. 
Consider now a matrix 
$$
\mathcal{C}=
\begin{bmatrix}
A_j \ast A_i
\end{bmatrix}
_{1\le i,j \le m}.
$$
Observe that $\det\mathcal{B}(1)=\prod_{i=1}^m Pr(A_i)\det\mathcal{C}$ and $\det\mathcal{B}^j(1)=\prod_{i=1}^m Pr(A_i)\det\mathcal{C}^j$, where $\mathcal{C}^j$
is the matrix formed by replacing the $j$-th column of $\mathcal{C}$ by  the column vector 
$\begin{bmatrix} 1\end{bmatrix}_{1\le i \le m}$.  In this way we can rewrite Corollaries  \ref{wn1} and  \ref{wn2} in terms of matrix $\mathcal{C}$ as follows 

$$
Pr(\tau=\tau_i)=\frac{\det \mathcal{C}^i}{\sum_{j=1}^m \det \mathcal{C}^j}\quad {\rm and}\quad E\tau=\frac{\det \mathcal{C}}{\sum_{j=1}^m \det \mathcal{C}^j}.
$$
By the martingale arguments Li proved  in  \cite{Li} (see Theorem 3.1) that for every $i=1,2, \ldots, m$ we have
$$
E\tau= \sum_{j=1}^{m} Pr(\tau=\tau_j)( A_j \ast A_i)
$$
which in our notation is equivalent to the following formula
$$
\det \mathcal{C} = \sum_{j=1}^{m}  (A_j \ast A_i) \det \mathcal{C}^j, \quad (1\leq i \leq m).
$$
Now we independently prove  the above determinant's identity. 
\begin{proof}
Let $\mathcal{D}$ be extended matrix $\mathcal{C}$ on an initial row and  column as follows
$$
\mathcal{D}=
\left[\begin{array}{c|c}
a_{00}& \begin{array}{cccc} a_{01} & a_{02} & \ldots & a_{0m} \end{array} \\
\hline
\begin{array}{c} a_{10} \\ a_{20} \\ \vdots \\ a_{m0} \end{array}&\mathcal{C}

\end{array}\right],
$$
where $a_{00}=1$, $(a_{01}, \ldots, a_{0m})$ is the $i$-th row of the matrix $\mathcal{C}$, i.e. $a_{0j}=A_j \ast A_i$, $1\le j \le m$,  and $a_{k0}=1$ for $k\neq i$ and $a_{i0}=0$. 
The Laplace expansion along the zero column yields 
$$
\det \mathcal{D} = \det \mathcal{C}
$$
whereas taking the Laplace expansion along the $i$-th row we obtain
$$
\det \mathcal{D}  = \sum_{j=1}^{m}(A_j \ast A_i) (-1)^{i+j}\det\mathcal{D}_{ij}.
$$
Notice that for every matrix $\mathcal{D}_{ij}$
permuting the zero column and zero row in the place of removed ones, by the determinant's properties, we get
$$
\det \mathcal{D}_{ij}=(-1)^{i+j-2}\det \mathcal{C}^j.
$$ 
Thus 
$$
\det \mathcal{D}  = \sum_{j=1}^{m}(A_j \ast A_i)\det \mathcal{C}^j
$$
and this completes the proof.
\end{proof}
\end{rem}

\begin{exa}

Let $\Omega=\{A,C,G,T\}$ and $(\xi_n)$ be a sequence of i.i.d. letters in $\Omega$  with the probabilities 
$$
Pr(\xi_n=A)=p_a, \quad Pr(\xi_n=C)=p_c, \quad Pr(\xi_n=G)=p_g\quad {\rm and} \quad P(\xi_n=T)=p_t.
$$
Consider the set of three  patterns: $A_1=ACG, A_2=ATG$ and $A_3=AG$. Observe that $w_{A_i}^{A_j}(s)=0$ if $i\neq j $ and $w_{A_i}^{A_i}(s)=1$.
So in this case the matrix $\mathcal{B}(s)= [w_{A_i}^{A_j}(s)]_{1\leq i,j \leq 3}$ is the identity matrix. For the matrices $\mathcal{B}^i(s)$ we obtain
$$
\det \mathcal{B}^i(s)= P(A_i)s^{l_i}.
$$
Since $Pr(\tau=\tau_i)= \frac{det \mathcal{B}^i(1)}{\sum_{j=1}^{m}det \mathcal{B}^j(1)}$ the probabilities that the $i$-th pattern occurs as a first are given by
$$
Pr(\tau=\tau_1)=\frac{p_c}{p_c+p_t+1}, \ \ \ \ \ 
Pr(\tau=\tau_2)=\frac{p_t}{p_c+p_t+1},  \ \ \ \ \
Pr(\tau=\tau_3)=\frac{1}{p_c+p_t+1}. 
$$
Applying now Colorary \ref{wn2} we obtain
$$
E\tau = \frac{1}{p_ap_g(p_c+p_t+1)}.
$$
By Proposition \ref{stw} we can calculate 
\begin{eqnarray}
E(\tau|\tau=\tau_1)&=& \frac{p_ap_g+1}{p_ap_g(p_c+p_t+1)}, \nonumber \\
E(\tau|\tau=\tau_2)&=& \frac{p_ap_g+1}{p_ap_g(p_c+p_t+1)},  \nonumber \\
E(\tau|\tau=\tau_3)&=& \frac{1-p_ap_g(p_c+p_t)}{p_ap_g(p_c+p_t+1)}. \nonumber
\end{eqnarray}
Note that for each pattern by the Solov'ev's formula (\ref{solov}) we have 
$$
E\tau_1=\frac{1}{p_ap_cp_g}, \ \ \ \ \ 
E\tau_2= \frac{1}{p_ap_gp_t}, \ \ \ \ \  
E\tau_3= \frac{1}{p_ap_g}. 
$$
Observe that for different $p_a,p_c,p_g,p_t$ the values $E\tau_1,E\tau_2,E\tau_3$  are mostly different. But in this example we obtained
that $E(\tau|\tau=\tau_1)=E(\tau|\tau=\tau_2)$ for any $p_a,p_c,p_g,p_t$.
\end{exa}

\begin{rem}
If $\mathcal{B}(s)$ is the identity matrix then applying Proposition \ref{stw} one can calculate that
\begin{eqnarray}
E(\tau|\tau=\tau_i)&=& E\tau + \frac{l_i Pr(A_i)\sum_{k=1}^{m}Pr(A_k)- Pr(A_i)\sum_{k=1}^{m} l_k Pr(A_k)}{Pr(A_i)\sum_{k=1}^{m}Pr(A_k)}  \nonumber \\
     \;               &=& E\tau +  l_i - \frac{\sum_{k=1}^{m}l_k Pr(A_k)}{\sum_{k=1}^{m} Pr(A_k)}.                                     \nonumber 
\end{eqnarray} 
It means that if additionally patterns $A_i$ and $A_j$ are of the same length then $E(\tau|\tau=\tau_i)=E(\tau|\tau=\tau_j)$. 
For this reason in the above example $E(\tau|\tau=\tau_1)=E(\tau|\tau=\tau_2)$.
\end{rem}

\begin{exa}
Let now $\Omega=\{H,T\}$ and $Pr(\xi_n=H)=p$ and $Pr(\xi_n=T)=q=1-p$.
Consider three patterns: $A_1=THH, A_2=HTH$ and $A_3=HHT$.
In this case
$$
\mathcal{B}(s)= \begin{pmatrix} w_{A_i}^{A_j}(s)\end{pmatrix}_{1\leq i,j \leq 3}= 
\begin{pmatrix}
	1        &  ps   & p^2s^2 \\
	pqs^2 &   pqs^2+ 1      & ps \\
	pqs^2+qs        &    pqs^2  & 1
\end{pmatrix}
$$
and
\begin{eqnarray*}
\det\mathcal{B}(s) & = &-p^3q^2 s^5-2p^2qs^3+pqs^2+1,\\
\det\mathcal{B}^1(s) & = & p^2qs^3(-p^2qs^3+pqs^2-ps+1),\\
\det\mathcal{B}^2(s) & = & p^2qs^3(1-ps),\\
\det\mathcal{B}^3(s) & = & p^2qs^3(-pq^2s^3-qs+1).
\end{eqnarray*}
By Corollary \ref{wn1} we get
$$
Pr(\tau=\tau_1)=\frac{q(1+pq)}{1+q}, \ \ 
Pr(\tau=\tau_2)=\frac{q}{1+q}, \ \
Pr(\tau=\tau_3)=\frac{p(1-q^2)}{1+q}.  
$$
Taking into account that $q=1-p$, by virtue of Corollary \ref{wn2} we obtain the formula on the expected waiting time of $\tau$ 
\begin{eqnarray}
E\tau &=&\frac{-p^5+2p^4+p^3-3p^2+p+1}{p^2(1-p)(2-p)}. \nonumber 
\end{eqnarray}
By Proposition \ref{stw} one can calculate that
\begin{eqnarray}
E(\tau|\tau=\tau_1)&=& E\tau+ \frac{-p^5+9p^4-12p^3-6p^2+12p-3 }{(1-p)(1+p-p^2)(2-p)}, \nonumber \\
E(\tau|\tau=\tau_2)&=& E\tau+ \frac{p^3-10p^2+11p-3}{(1-p)(2-p)}, \nonumber \\
E(\tau|\tau=\tau_3)&=& E\tau+ \frac{-p^4+ 7p^3+3p^2-10p+3}{p^2(2-p)}. \nonumber
\end{eqnarray}
In the symmetric case $p=q=\frac{1}{2}$ the above quantities take values
$$
Pr(\tau=\tau_1)=\frac{5}{12},\quad 
Pr(\tau=\tau_2)=\frac{1}{3},\quad
Pr(\tau=\tau_3)=\frac{1}{4},  
$$
$$
E\tau=\frac{31}{6} 
$$
and
$$
E(\tau|\tau=\tau_1)=\frac{86}{15},\quad E(\tau|\tau=\tau_2)=\frac{16}{3},\quad E(\tau|\tau=\tau_3)=4.
$$

\end{exa}

\end{document}